\documentclass{amsart}
\usepackage[utf8]{inputenc}
\usepackage{amsmath,amssymb,mathtools, amsxtra}
\usepackage{tikz, tikz-cd}
\usepackage{color}
\usepackage[all]{xy}
\usepackage{caption}
\usetikzlibrary{calc, intersections, through}
\usepgflibrary{arrows}
\usetikzlibrary{positioning}
\usepgflibrary{arrows}
\usepackage{enumitem}
\usepackage{mathtools}

\tikzstyle{wbullet}=[circle, draw=black, fill=white, thick, inner sep=0pt, minimum size=1.5mm]
\tikzstyle{bbullet}=[circle, draw=black, fill=black, inner sep=0pt, minimum size=1.5mm]

\usepackage{array}
\usepackage{pdflscape}
\usepackage{afterpage}
\usepackage{capt-of}
\usepackage{lipsum}
\usepackage{graphicx,caption}
\usepackage{bbm}
\usepackage[symbol]{footmisc}
\usepackage{xcolor}

\usepackage[
 hypertexnames=false,
pdfpagemode=UseNone,
breaklinks=true,
extension=pdf,
urlcolor=blue,
colorlinks=true,
linkcolor=blue,
citecolor=blue,
]{hyperref}
\hypersetup{pdfauthor={Name}}

 \usepackage{todonotes}

\newtheorem{thm}{Theorem}[section]

\newtheorem{lem}[thm]{Lemma}
\newtheorem{lem-defn}[thm]{Lemma-Definition}

\newtheorem{prop}[thm]{Proposition}
\newtheorem{cor}[thm]{Corollary}
\theoremstyle{definition}
\newtheorem{defn}[thm]{Definition}
\newtheorem{nota}[thm]{Notation}

\newtheorem{rmk}[thm]{Remark}

\newtheorem{ex}[thm]{Example}
\theoremstyle{remark}

\numberwithin{equation}{section}

\newcommand{\op}{\mathrm}

\newcommand{\CC}{\mathbb{C}}

\newcommand{\QQ}{\mathbb{Q}}

\newcommand{\ZZ}{\mathbb{Z}}

\newcommand{\hG}{\widehat{G}}
\newcommand{\hH}{\widehat{H}}

\newcommand{\sE}{\mathcal{E}}
\newcommand{\sF}{\mathcal{F}}
\newcommand{\sH}{\mathcal{H}}

\newcommand{\sO}{\mathcal{O}}

\newcommand{\sU}{\mathcal{U}}

\newcommand{\sZ}{\mathcal{Z}}

\newcommand{\gp}{\mathfrak{p}}

\newcommand{\Aut}{\mathrm{Aut}}

\newcommand{\ch}{\mathrm{ch}}
\newcommand{\codim}{\mathrm{codim}}
\newcommand{\even}{\mathrm{even}}

\newcommand{\id}{\mathrm{id}}

\newcommand{\Ind}{\mathrm{Ind}}

\newcommand{\pt}{\mathrm{pt}}

\newcommand{\reg}{\mathrm{reg}}
\newcommand{\Res}{\mathrm{Res}}

\newcommand{\rk}{\op{rk}}

\newcommand{\Supp}{\mathrm{Supp}}

\newcommand{\td}{\mathrm{td}}

\newcommand{\Tr}{\mathrm{Tr}}

\title[Chevalley--Weil formula for  higher dimensional compact complex manifolds]{
The Chevalley--Weil formula for finite group actions on higher dimensional compact complex manifolds}
 \author{ Wenfei Liu}
 \address{Wenfei Liu \\School of Mathematical Sciences\\Xiamen University\\Xiamen, Fujian 361005\\P.R.~China}
 \email{wliu@xmu.edu.cn}

  \author{Renjie Lyu}
 \address{Renjie Lyu \\School of Mathematical Sciences\\Xiamen University\\Xiamen, Fujian 361005\\P.R.~China}
 \email{r.lyu@xmu.edu.cn}
 
\date{\today}
\begin{document}

\begin{abstract}
Building on the Atiyah--Singer holomorphic Lefschetz fixed-point theorem, we define ramification modules associated to the fixed loci of a finite group acting on a compact complex manifold. This allows us to generalize the  Chevalley--Weil formula for compact Riemann surfaces to higher dimensions. More precisely, let $G$ be a finite group acting on a compact complex manifold $X$, and let $\sE$ be a $G$-equivariant locally free sheaf on $X$. Then, in the representation ring $R(G)_\QQ$, we have
\[
\chi_G(X, \sE):=\sum_{i=0}^{\dim X}(-1)^i[H^i(X, \sE)]=\frac{1}{|G|}\chi(X,\sE)[\CC[G]] + \sum_Z\Gamma(\sE)_Z
\]
where $Z$ runs over all connected components of the fixed-point sets $X^g$ for $g\in G$, and each $\Gamma(\sE)_Z\in R(X)_\QQ$, called the \emph{ramification module} at $Z$, depends only on the restriction $\sE|_Z$ and the normal bundle $N_{Z/X}$ as $G_Z$-equivariant bundles. We illustrate the computation of $\Gamma(\sE)_Z$ in several special cases and provide a detailed example for faithful actions of $G\cong(\ZZ/2\ZZ)^n$ on a compact complex surface.
\end{abstract}

\maketitle
\tableofcontents

\section{Introduction}
Let $X$ be a compact Riemann surface, and $G\subset\Aut(X)$ a finite subgroup of automorphisms. The classical Chevalley--Weil formula (\cite{CW34, Wei35}) expresses the $G$-module $H^0(X, \omega_X^{\otimes n})$ of $n$-differentials as a rational multiple of the regular representation $\CC[G]$, plus a correction term determined by the branched locus of the quotient map $X\rightarrow X/G$. This result was generalized  in \cite{EL80} to a formula for more general $G$-equivariant locally sheaves on smooth projective tame $G$-curves over an arbitrary algebraically closed field. Subsequent work has produced many further refinements with an arithmetic flavor; see \cite{LL25} for a relatively complete list of references. Recently, the Chevalley--Weil formula for proper singular curves has been established in \cite{Ton25, LL25}.

As explicitly illustrated in \cite{Koc05} and ~\cite{Ara22}, the Chevalley--Weil formula for \break $G$-equivariant locally free sheaves on a compact Riemann surface can be deduced from a more general fixed-point formulism (\cite{AS68III}). In this paper, we extend this framework to higher dimensions.
 Unlike the approach in \cite{EL80}, which relies on the quotient map $X\rightarrow X/G$, our method applies the Atiyah–Singer holomorphic Lefschetz fixed-point theorem directly to capture the contributions from the fixed loci via the so-called ramification modules; see Definition~\ref{defn: Gamma_Z}.

\begin{thm}\label{thm: main}
 Let $X$ be a compact complex manifold, possibly disconnected and non-equidimensional, and let $G$ be a finite group acting on $X$. Let $\sE$ be a $G$-equivariant locally free sheaf on $X$. Let $\chi_G(X, \sE):=\sum_j (-1)^j [H^j(X, \sE)]$ be the $G$-Euler characteristic of $\sE$,
where $[H^j(X, \sE)]$ denotes the class of $H^j(X, \sE)$ in the Grothendick group $R(G)$ of $G$-modules. Then we have an equality in $R(G)_\QQ$:
\begin{equation}\label{eq: CW main thm}
  \chi_G(X, \sE):=\sum_{i=0}^{\dim X}(-1)^i[H^i(X, \sE)] = \frac{1}{|G|}\chi(X,\sE)[\CC[G]]+ \sum_Z\Gamma(\sE)_Z  
\end{equation}
where $Z$ runs through the components of the fixed loci $X^g$ with $g\in G$, and $\Gamma(\sE)_Z\in R(G)_\QQ$ is the ramification module at $Z$, given in Definition~\ref{defn: Gamma_Z}.
\end{thm}
One has $\chi_G(X,\sE)=\frac{1}{|G|}\chi(X,\sE)[\CC[G]]$ if the action of $G$ on $X$ is free. Thus, the ramification modules $\Gamma(\sE)_Z$ can be viewed as correction terms arising from the deviation of the $G$-action being free. 

The key part in proving Theorem~\ref{thm: main} is the definition of the ramification modules \(\Gamma(\sE)_Z\). For $g\in G$, let $H:=\langle g\rangle$ be the cyclic group it generates, and let $X^H$ be the fixed locus of $H$. Let $N^*:=N_{X^H/X}^*$ be the conormal bundle of $X^H$, and let $\lambda_{-1}N^* = \sum_{i\geq 0}[\wedge^i N^*]\in K_{H}(X^H)$ the fundamental class of $N^*$ in the $H$-equivariant $K$-group of $X^H$. Then Atiyah--Singer's holomorphic Lefschetz fixed-point theorem (Theorem~\ref{thm: AS}) computes the trace on the $G$-Euler characteristic:
\begin{equation}
\label{eqs:AS-Lefschetz-trace-formula}
\Tr(g; \chi_G(X, \sE)) = \Tr\left(g;\int_{X^H} \frac{\ch_H(\sE|_{X^H})\td(X^H)}{\ch_H(\lambda_{-1}N^*)}\right).
\end{equation}
Here, $\ch_H$ is the $H$-Chern character (Definition~\ref{defn: H-Chern}): 
\[
K_H(X^H) = K(X^H)\otimes R(H)\xrightarrow{\ch_H} H^{\even}(X^H, \ZZ)\otimes R(H).
\]
Let \(\gp_g\coloneqq \ker(\Tr_g\colon R(H)\rightarrow \CC)\) be the prime ideal 
of virtual \(H\)-modules with vanishing trace at \(g\), and let \(R(H)_g\) be
the localization of $R(H)$ at \(\gp_g\). The trace \(\Tr(g; \lambda_{-1}N^*)\) is non-zero, so $\lambda_{-1}N^*$ is a unit in $R(H)_g$.
The trace formula~\eqref{eqs:AS-Lefschetz-trace-formula} then implies that the integral expression
\(\int_{X^H} \frac{\ch_H(\sE|_{X^H})\td(X^H)}{\ch_H(\lambda_{-1}N^*)}\) 
is the image of \(\chi_G(X, \sE)\) in \(R(H)_g\) under the homomorphism 
\[
R(G)\xrightarrow{\Res^G_H} R(H)\xrightarrow{\mathrm{loc}_g} R(H)_g,
\]
where \(\Res^G_H\) is the restriction map and \(\mathrm{loc}_g\) is the localization at $\gp_g$. Thus, the problem reduces to recovering the global $G$-Euler characteristic $\chi_G(X,E)$ from its localized restrictions to cyclic subgroups.

Using Artin's theorem in Section~\ref{subsec:artin's theorem}, one can theoretically recover the virtual \(G\)-module \(\chi_G(X, \sE)\) from its restrictions to all cyclic subgroups of \(G\) (Proposition~\ref{prop: recover chi}):
\begin{equation}\label{eq: recover chi_G from restriction}
\chi_G(X, \sE) = \sum_{H\subset G \text{ cyclic}} \frac{|H|}{|G|}\Ind_H^G(\theta_H\otimes \Res_H^G\chi_G(X, \sE)),
\end{equation}
where $\theta_H$ is the characteristic module on the generators of $H$ (see Definition~\ref{defn: theta_H}). 

Our goal is to reformulate the \(H\)-module \(\theta_H\otimes \Res_H^G\chi_G(X, \sE)\) in terms of ramification modules, which are characterized by the localized fixed-point contributions. For any cyclic subgroup \(H\subset G\) and each component \(Z\) in \(X^H\), we show in Lemma~\ref{eqs:theta_ZH} the existence of an element \(\tau_{Z,H}\in R(H)_\QQ\) such that
\[
\theta_H\otimes \Res_H^G\chi_G(X, \sE) = \sum_{Z}\int_Z\theta_H\tau_{Z,H}\ch_H(\sE|_{Z})\td(Z).
\]
This leads to Definition~\ref{defn: Gamma_Z} for the ramification module 
\[
\Gamma(\sE)_Z=\sum_{H\in \sH_Z} \frac{|H|}{|G|}\Ind_H^G \left(\theta_H \int_Z\ch_H(\sE|_Z)\tau_{Z,H}\td(Z)\right),
\]
where \(\mathcal{H}_Z\) is the set of cyclic subgroups \(H\subset G\) for which \(Z\) is a component of \(X^H\). The equality in Theorem~\ref{thm: main} is then a direct consequence of this definition.

The Chevalley--Weil formula~\eqref{eq: CW main thm} 
facilitates the computation of the multiplicity $\mu_M \chi_G(X, \sE)$ of an irreducible $G$-module $M$ in
the virtual $G$-module $\chi_G(X, \sE)$
(Corollary~\ref{cor: multiplicity}). 
Its practical utility depends on the computability of the ramification modules $\Gamma(\sE)_Z$. In Section~\ref{sec: ramification module}, we present explicit forms of these ramification modules \(\Gamma(\sE)_Z\) under one of the following conditions:
\begin{enumerate}
    \item The stabilizer $G_Z$ is cyclic (Lemma~\ref{lem: cyclic intertia}).
    \item $\codim_X(Z)\leq 1$ or $\dim Z=0$ ((Examples~\ref{ex: codim 0}--\ref{ex: codim 1}).
\end{enumerate}
In Section~\ref{sec: surface}, we determine $\chi_G(X, \sE)$ explicitly when $G\cong (\ZZ/2\ZZ)^n$ and $X$ is a compact complex surface.  In this case, both conditions (i) and (ii) are easily verified.

\medskip

\noindent{\bf Notation and Conventions.} 
Let $X$ be a set endowed with an action of a group $G$. We call $X$ a $G$-set. For a subset  $Z\subset X$, the subgroup 
\[
G_Z:=\{g\in G\mid g(x)=x,\,\forall\, x\in Z\}
\]
is called the \emph{(pointwise) stabilizer} of $Z$. 
For a subset $K\subset G$, the \emph{fixed locus} of $K$ is
\[
X^K:=\{x\mid g(x)=x, \, \forall\, g\in K\}.
\]
We denote the neutral element of $G$ by  $\id_G$ or $\id$, and the order of $G$ by $|G|$. For $g\in G$, we use $|g|$ to denote its order. 

A $G$-manifold is a $G$-set $X$ endowed with a manifold structure preserved by the $G$-action. Similarly, one defines a $G$-vector space, a complex $G$-manifold, and so on, by imposing a  structure  compatible with the $G$-action. A quasi-coherent \emph{$G$-sheaf} (or \emph{$G$-equivariant sheaf}) on a complex $G$-manifold $X$ means a quasi-coherent sheaf $\sF$ together with isomorphisms $\Phi_g\colon g^*\sF\rightarrow \sF$ for each $g\in G$ satisfying the following compatibility condition (\cite{Bri25}):
\[
\Phi_{gh} = \Phi_h\circ h^*(\Phi_g),\quad  \text{for all }\, g,h\in G.
\]


Let $R$ be a commutative ring with unit. For two $R$-modules $M$ and $N$, their tensor product $M\otimes_R N$ over $R$ is simply denoted by $M\otimes N$ if $R$ is clear from the context. Suppose that $A$ and $B$ are two commutative $R$-algebras with units and that one of them is free as a $R$-module. Then $\iota_A\colon A\rightarrow A\otimes_R B,\,a\mapsto a\otimes 1$ and $\iota_B\colon B\rightarrow A\otimes_R B,\, b\mapsto 1\otimes b$ are injective ring homomorphisms, and we may view $A$ and $B$ as sub-$R$-algebras of $A\otimes_R B$ via $\iota_A$ and $\iota_B$. For $a\in A$ and $b\in B$, we will often write $ab$ or $a\cdot b$ for $a\otimes b$, which will result in no ambiguity. For a $\ZZ$-module $M$ and a $\ZZ$-algebra $K$, We usually denote $M\otimes_\ZZ K$ by $M_K$.

\medskip

\noindent{\bf Acknowledgment} W.L. would like to thank Qing Liu for many helpful discussions on the Chevalley--Weil formula of curves during the cooperation on the paper \cite{LL25}; the idea of defining the ramification module originates from there. He would also like to thank Lie Fu for his interest on the subject and for pointing out the notion of inertia stack, which inspires our definition of strata on a $G$-manifold. This research project has been supported by the NSFC (No.~12571046).

\section{Preliminaries}
\subsection{Constructions around the representation rings}
The material of this subsection is based on \cite{Ser77}, \cite{Seg68} and \cite[Section~1]{AS68II}; see also \cite[Section~2.1]{LL25}.

Let $G$ be a finite group and $k$ a field. The \emph{Grothendieck ring} $R_k(G)$ of $k$-representations of $G$, also called the \emph{representation ring of $G$ over $k$}, is the abelian group of formal finite $\ZZ$-linear combinations $\sum_i n_i V_i$ of $k$-representations $V_i$ of $G$, modulo the relations $V'+V''-V$ for every $G$-equivariant exact sequence $0\rightarrow V'\rightarrow V\rightarrow V''\rightarrow 0$. Since we primarily deal with complex representations, we denote $R_\CC(G)$ by $R(G)$, and a representation without a specified base field means a complex representation.

The class of a representation $V$ of $G$ in $R(G)$ is denoted by $[V]_G$ (or $[V]$ if the group is clear). Elements of $R(G)$ are called \emph{virtual $G$-modules}. The ring structure on $R(G)$ is induced by the tensor product, with $([V],[W])\mapsto [V\otimes W]$, and the multiplicative unit is  $[1_G]$, the class of the one-dimensional trivial representation of $G$. The classes of one-dimensional irreducible representations form a multiplicative subgroup of $R(G)$, and we can talk about the order of an element in this subgroup. To avoid confusion with other tensor products, we denote the multiplication of $\alpha, \beta\in R(G)$ by $\alpha\cdot \beta$ or $\alpha\beta$. For example, if $\alpha=[V]$ and $\beta=[W]$, then $\alpha\cdot \beta = [V\otimes W]$.

The underlying additive abelian group of $R(G)$ is free, with a basis given by the irreducible representations of $G$. Consequently, the natural homomorphisms $R(G)\rightarrow R(G)_\QQ \rightarrow R(G)_\CC$ are injective. A fundamental result is the ring isomorphism $R(G)_\CC\rightarrow C(G)$, which sends $\alpha\in R(G)_\CC$ to its character function $c(\alpha)\colon g\mapsto \Tr(g; \alpha)$, where $C(G)$ is the $\CC$-vector space of class functions on $G$. Since $C(G)$ has a trivial nilradical, so do $R(G)_\CC$ and its subrings $R(G)$ and $R(G)_\QQ$. The inner product of two elements $\alpha,\beta\in R(G)_\CC$ is defined via their characters:
\[
\langle \alpha,\beta\rangle_G:=\langle c(\alpha), c(\beta)\rangle_G:=\frac{1}{|G|}\sum_{g}\Tr(g;\alpha)\cdot \overline{\Tr(g;\beta)}.
\]
The support of $\phi\in R(G)_\CC$ is defined as 
\begin{equation}
    \Supp(\phi):=\{g\in G\mid \Tr(g;\phi)\neq 0\}.
\end{equation}

We recall several standard maps associated with $R(G)$:
\begin{itemize}
    \item For a subgroup $H\subset G$, we have the \emph{restriction map} $\Res^G_H\colon R(G)\rightarrow R(H)$ and the \emph{induction map} $\Ind_H^G\colon R(H)\rightarrow R(G)$.
    \item For $g\in G$, the trace defines a ring homomorphism 
\[
\Tr_g\colon R(G)\rightarrow \CC,\quad \alpha\mapsto \Tr(g; \alpha)
\]
Since $\CC$ is an integral domain, the kernel $\gp_g:=\ker \Tr_g$ is a prime ideal of $R(G)$, and we denote by $R(G)_g:=R(G)_{\gp_g}$ the localization at $\gp_g$. 
For $g=\id$, the prime ideal $I_G:=\gp_{\id}$ is called the \emph{augmentation ideal} of $R(G)$. Note that, $R(G)$ can have zero divisors,\footnote{For example, $\chi_\reg\cdot  I_G =0$, where $\chi_\reg$ is the the class of the regular representation.} 
and the natural map $R(G)\rightarrow R(G)_g$ is in general neither injective nor surjective.
\end{itemize}

\begin{lem}\label{lem: image to localization}
 Suppose that $G=\langle g_0\rangle$ is a nontrivial
 finite cyclic group, and let $\phi\in \hG\setminus\{1_G\}$ be a non-trivial irreducible representation. Let $K=\{g\in G\mid \phi(g)=1\}$. Then, for $\psi\in R(G)$ such that $\Tr(g;\psi)=0$ for all $g\in K$, the element $\frac{|\phi|\psi}{1-\phi}\in R(G)_{g_0}$ is the image of $-\psi\cdot \sum_{d=0}^{|\phi|-1}d\phi^d\in R(G)$ under the localization map \(R(G)\to R(G)_{g_0}\), where $|\phi|$ denotes the order of the character $\phi$.
\end{lem}
\begin{proof}
Using \cite[Lemma~1.2]{Koc05} and the assumption that $\Tr(g;\psi)=0$ for $g\in K$, a direct calculation shows that
\[
\Tr\left(g;\,\psi\cdot \sum_{d=0}^{|\phi|-1}d\phi^d(1-\phi)\right) = -|\phi|\Tr(g;\psi) ,\quad \text{for all } g\in 
G.
\]
The result follows from this identity.
\end{proof}

\subsection{Recovering a representation from its restrictions to cyclic subgroups}
\label{subsec:artin's theorem}
Let $G$ be a finite group. By Artin's theorem \cite[9.2]{Ser77}, every element $\chi\in R(G)$ is a $\QQ$-linear combination of characters induced from cyclic subgroup of $G$. 
 In this subsection, we provide a concrete expression of this fact in Proposition~\ref{prop: recover chi}.

\begin{defn}\label{defn: theta_H}
For a cyclic group $H$, we define an element $\theta_H\in \frac{1}{|H|}R(H)$ such that (\cite[Section~9.4]{Ser77})\footnote{
{Note that our definition of $\theta_H$ differs from that in \cite[Section~9.4]{Ser77} by a factor of $\frac{1}{|H|}$.}
}:
\begin{equation}
   \Tr(h; \theta_H)=
    \begin{cases}
        1 & \text{if $h$ generates $H$,}\\
        0 & \text{otherwise.}
    \end{cases}
\end{equation}
Thus, $\Tr(\cdot;\theta_H)$ is the characteristic function on the set of generators of $H$. 
We call $\theta_H$ the \emph{characteristic module} supported
on the generators of $H$.
\end{defn}

For any finite group \(G\), it is known by~\cite[Prop. 27]{Ser77} that 
\begin{equation}
\label{eqs:induced-character-funs}
|G|\cdot [1_G]=\sum_{B\subset G} |B|\cdot \Ind^{G}_{B} \theta_B
\end{equation}
where \(B\) runs over all the cyclic subgroups of \(G\).
In particular, for any cyclic group \(H\), we obtain
\begin{equation}
\label{eqs:theta_H-induction}
\theta_H=[1_H]-\sum_{B\subset H} \frac{|B|}{|H|}\cdot \Ind^{H}_{B} \theta_B
\end{equation}
where \(B\) runs over all the proper cyclic subgroups of \(H\). This shows that \(\theta_H\) belongs to \(\frac{1}{|H|}R(H)\). If \(H\) has prime order, then 
\[
\theta_H=[1_H]-\Ind^{H}_{\id} \theta_{\id}=[1_H]-\frac{1}{|H|}[\CC[H]].
\] 
Note that, for a nontrivial cyclic group \(H\), the trace \(\Tr(1; \theta_H)=0\), meaning \(\theta_H\) lies in the augmentation ideal \(I(H)_\QQ\).

Using $\theta_H$, the following proposition recovers a virtual $G$-module from its restriction to the cyclic subgroups of $G$.
\begin{prop}\label{prop: recover chi}
Let $G$ be a finite group, and let $\chi\in R(G)_\QQ$ be a virtual \(G\)-module with rational coefficients.
Then 
\begin{equation}\label{eq: recover chi}
\chi = \sum_{H\subset G} \frac{|H|}{|G|}\Ind_H^G(\theta_H\cdot\Res_H^G \chi) \in R(G)_\QQ
\end{equation}
where the sum runs over all the cyclic subgroups $H$ of \(G\).
\end{prop}
\begin{proof}
    Recall formula~\eqref{eqs:induced-character-funs}:
    \[
    [1_G]=\frac{|H|}{|G|}\sum_{H\subset G}\Ind^{G}_{H} \theta_H,
    \]
   Multiplying this equation by $\chi$ yields
    \[
    \chi=\frac{|H|}{|G|}\sum_{\substack{H\subset G\\ \text{cyclic}}}\Ind^{G}_{H} \theta_H\cdot \chi
    =\sum_{\substack{H\subset G\\ \text{cyclic}}}\frac{|H|}{|G|}\Ind^{G}_{H} (\theta_H\cdot \Res^G_H\chi),
    \]
where the second equality follows from \cite[Section~3.3, Example 5]{Ser77}.
\end{proof}

\subsection{\texorpdfstring{$G$}{G}-equivariant \texorpdfstring{$K$}{K}-theory and the holomorphic Lefschetz fixed-point theorem}
Let $G$ be a finite group and \(X\) be a compact complex $G$-manifold, probably disconnected and non-equidimensional.
\footnote{For example, the fixed locus $M^\sigma$ for a finite-order automorphism $\sigma$ of a connected compact complex manifold $M$ can be disconnected and non-equidimensional.}
Let $K_G(X)$ be the Grothendieck ring of $G$-equivariant locally free sheaves (of finite rank) on $X$. If $G$ is trivial, we drop the subscript and write $K(X)$.

Suppose we have a homomorphism of finite groups $\varphi\colon H\rightarrow G$, a compact complex $H$-manifold $Z$, and a map $f\colon Z\rightarrow X$ such that $f(h(z))= \varphi(h)(f(z))$ for all $h\in H,\, z\in Z$. Then pulling back locally free sheaves induces a natural ring homomorphism:
\[
\Res^G_H\colon K_G(X)\rightarrow K_H(Z),\quad [\sE]\mapsto [f^*\sE]
\]
This applies in particular to the inclusion of a subgroup $H\subset G$.

The map \(f: X\to \pt\) to a point induces a pullback ring homomorphism 
\[
f^*\colon R(G)=K_G(\pt)\rightarrow K_G(X),
\]
making $K_G(X)$ into a $R(G)$-algebra. For a locally free $G$-sheaf $\sE$ on $X$, its image under the Gysin map \(f_!: K_G(X)\to K_G(\pt)\cong R(G)\) is the $G$-Euler characteristic of $\sE$:
\[
\chi_G(X,\sE):=f_![\sE] = \sum_{i= 0}^{\dim X}(-1)^i[H^i(X,\sE)]\in R(G).
\]

Now suppose the action of $G$ on $X$ is trivial. Then there is a ring isomorphism (\cite[Proposition~2.2]{Seg68})
\begin{equation}
\label{eqs:G-equiv-K-grp-iso}
\nu: K_G(X)\to K(X)\otimes R(G),\quad
\nu(\sE)=\sum_{M\in \hG} \mathrm{Hom}_{G}(M, \sE)\otimes [M],
\end{equation}
where $\hG$ denotes the set of isomorphism classes of irreducible $G$-modules, and $\mathrm{Hom}_G(M, \sE)$ denotes the sheaf
of $G$-equivariant homomorphisms.  
Via this isomorphism, we have a ring homomorphism (\cite[1.5.4]{Ill68})
\begin{equation}
\label{eq:chern-character-hom}
\ch_G \colon K_G(X)\cong K(X)\otimes R(G)\to H^{\even}(X, \ZZ)\otimes R(G),\quad u\otimes\phi \mapsto \ch(u)\phi
\end{equation}
where $H^{\even}(X, \ZZ)$ denotes the even-degree part of cohomology ring $H^*(X, \ZZ)$ and \(\ch\colon K(X)\rightarrow H^{\even}(X,\ZZ)\) is the usual Chern character. For a given \(g\in G\), 
taking the trace of $g$ yields a ring homomorphism
\[
\Tr_g\colon K(X)\otimes R(G)\to H^\even(X,\CC), \quad u\otimes \phi\mapsto \Tr(g;\phi)\ch(u)  
\]
There is also an integration map:
\[
\int_X\colon H^{\even}(X, \ZZ)\otimes R(G) \rightarrow R(G),\quad u\otimes \phi\mapsto \left(\int_X u\right)\phi.
\]

Next we state the compatibility of the maps introduced above.
\begin{lem}\label{lem: compatibility}
Let $G$ be a finite group acting trivially on a compact complex manifold $X$, which may be disconnected and non-equidimensional, and let $H\subset G$ be a subgroup. Then the following  diagrams commute:
\begin{enumerate}
    \item \[
    \begin{tikzcd}
        K_G(X)\ar[d, "f_!"'] \ar[r, "\nu"] & K(X)\otimes R(G) \ar[d, "f_!\otimes \id"]\\
        K_G(\pt)\ar[r, "\nu"] & K(\pt)\otimes R(G). 
    \end{tikzcd}
    \]
    \item \[
    \begin{tikzcd}
        K_G(X)\ar[r, "\nu"] \ar[d, "\mathrm{Res}^G_H"'] & K(X)\otimes R(G) \ar[d, "\id\otimes \mathrm{Res}^G_H"] \\
        K_H(X)\ar[r, "\nu"] & K(X)\otimes R(H) 
    \end{tikzcd}
    \]
    \item 
    \[
    \begin{tikzcd}
   K_G(X)\arrow[r, "\ch_G"] \arrow[d, "\Res^G_H"']&  H^\even(X,\ZZ)\otimes R(G)\arrow[d, "\id\otimes \Res^G_H"]\\
   K_H(X)\arrow[r, "\ch_H"] & H^\even(X,\ZZ)\otimes R(H)
    \end{tikzcd}
    \]
\end{enumerate}
\end{lem}
\begin{proof}
 (i) By definition of $f_!$, we have
\[
f_!(\sE) = \chi_G(X, \sE) = \sum_{[M]\in \hG}\chi(X, \mathrm{Hom}_G(M, \sE))\otimes [M] =\sum_{[M]\in \hG} f_!(\mathrm{Hom}_G(M, \sE))\otimes [M]  
\]

(ii)   We need to prove the identity
    \begin{equation}
    \label{eqs:frobenius-reciprocity}
    \sum_{N\in \hH} \mathrm{Hom}_H(N, \mathrm{Res}^G_H\sE)\otimes N=\sum_{M\in \hG} \mathrm{Hom}_G(M, \sE)\otimes \mathrm{Res}^G_HM.
    \end{equation}
where \(M\) (resp. \(N\)) runs over irreducible \(G\)-modules (resp.~\(H\)-modules). 

Note that
\[
\mathrm{Res}^G_HM=\sum_{N\in \hH} \mu_{MN}N, \quad \text{where }\mu_{MN}=\langle \mathrm{Res}^G_HM, N\rangle_H.
\]
Thus, the right-hand side of~\eqref{eqs:frobenius-reciprocity} equals
\[
\sum_{N\in \hH} \mathrm{Hom}_G(\sum_{M\in \hG} \mu_{MN}M, \sE)\otimes N.
\]
The Frobenius reciprocity
\[
\langle \mathrm{Res}^G_HM, N\rangle_H=\langle M, \mathrm{Ind}^G_HN\rangle_G
\]
implies that
\[
\mathrm{Ind}^G_HN=\sum_{M\in \hG} \mu_{MN} M.
\]
Hence
\[
\sum_{N\in \hH} \mathrm{Hom}_G(\sum_{M\in \hG} \mu_{MN}M, \sE)\otimes N=\sum_{N\in \hH} \mathrm{Hom}_G(\mathrm{Ind}^G_HN, \sE)\otimes N.
\]
This is equal to the left-hand side of \eqref{eqs:frobenius-reciprocity} due to the adjunction isomorphism:
\[
\mathrm{Hom}_H(N, \mathrm{Res}^G_H\sE)=\mathrm{Hom}_G(\mathrm{Ind}^G_HN, \sE).
\]

(iii) For any \(\epsilon\in K_G(X)\), write
\[
\nu(\epsilon) = \sum_i u_i\otimes\phi_i, \quad u_i\in K(X),\, \phi_i\in R(G).
\]
By part (ii), we have
    \[
    \left(\id\otimes \Res^G_H\right)\circ \nu(\epsilon)=\nu(\Res^G_H \epsilon).
    \]
    It follows that
    \begin{multline*}
    \ch_H(\Res^G_H \epsilon)= 
    \ch_H((\id\otimes \Res^G_H)\circ \nu(\epsilon)) =\ch_H\left( \sum_i u_i\otimes \Res^G_H\phi_i\right) \\
    = \sum_i\ch(u_i)\otimes \Res^G_H\phi_i
    = \Res^G_H\ch_G(\epsilon).
    \end{multline*}
\end{proof}
 
The main result of this paper is based on the following famous theorem of Atiyah and Singer.
\begin{thm}[Holomorphic Lefschetz fixed-point formula, {\cite[Theorem 4.6]{AS68III}}]\label{thm: AS}
Let $G$ be a finite group acting on a compact complex manifold $X$, probably disconnected and non-equidimensional, and let \(\sE\) be a locally free $G$-sheaf over \(X\). Let \(X^g\) be the fixed locus of an element \(g\in G\), and let \(N^g=N_{X^g/X}\) the normal bundle of \(X^g\) in \(X\). Then  the following holds:
\begin{equation}\label{eq: AS}
\Tr(g;{\chi_G(X, \sE)}) = \Tr\left(g;\int_{X^g}\frac{\operatorname{ch}_{\langle g\rangle}(\sE|_{X^g})\cdot \td(X^g)}{\ch_{\langle g\rangle}(\lambda_{-1}(N^g)^*) }\right),
\end{equation}
where $(N^g)^*$ denotes the dual of $N^g$.   
\end{thm}


\section{The Chevalley--Weil formula for compact complex manifolds}
\begin{nota}\label{nota: CW}
 Let $X$ be a compact complex manifold, possibly disconnected and non-equidimensional, and let $G$ be a finite group acting on $X$ via a homomorphism $\rho\colon G\rightarrow \Aut(X)$, which is not necessarily faithful. For brevity, we denote $g(x):=\rho(g)(x)$ for $g\in G$ and $x\in X$. Let $\sE$ be a locally free $G$-sheaf on $X$. Note that $\ker \rho$ might act non-trivially on $\sE$.   
\end{nota}
\begin{defn}\label{defn: strata}
A component of $X^g$ for some $g\in G$ is called a \emph{stratum} of the pair $(X, G)$.
\end{defn}
Let $\sZ$ be the set of all strata.
For two strata $Z_1, Z_2\in \mathcal{Z}$, if $Z_1\subsetneq Z_2$, their stabilizers satisfy $G_{Z_1}\supsetneq G_{Z_2}$. For each stratum $Z\in \sZ$, define
\[
\sH_Z:=\{H\subset G_Z \text{ cyclic}\mid Z \text{ is a component of } X^H\}
\]
It is clear that $\sH_Z$ is nonempty and its elements are subgroups of $G_Z$.
Let $\sH$ be the set of all cyclic subgroups of $G$. For each $H\in \sH$,
define
\[
\sZ_H:=\{Z\in \sZ\mid \text{$Z$ is a component of $X^H$}\}.
\]
The set $\sZ_H$ is empty if and only if $X^H$ is empty.

Now for $Z\in \sZ$ and a cyclic subgroup $H\subset G_Z$, the sheaf $\sE|_{Z}$ is an $H$-equivariant sheaf on $Z$, and we may decompose $\sE|_{Z}$ into eigensheaves with respect to the action of $H$:
\[
\sE|_{Z} = \bigoplus_{\phi\in \hH} \sE_{Z,H, \phi}
\]
where $\hH$ denotes the set of irreducible representations of $H$, and for $\phi\in \hH$, $\sE_{Z,H, \phi}$ is the eigen-subsheaf of $\sE|_{Z}$ with character $\phi$. 
\begin{defn}\label{defn: H-Chern}
The $H$-Chern character of $\sE|_Z$ is defined as: 
\begin{equation}\label{eq: ch_H E|Z}
\ch_H(\sE|_Z)
:=\sum_{\phi\in\hH} \ch(\sE_{Z,H, \phi})\otimes[\phi]\in H^{\even}(Z,\ZZ)\otimes R(H).  
\end{equation}
\end{defn} 
Let us recall the description of  \(\ch_H \lambda_{-1}N_{Z/X}^*\) given in~\cite[Section~3]{AS68III}. Consider the eigen decomposition of the conormal bundle \(N^*_{Z/X}\):
\begin{equation}
\label{eq:eigen-decomposition}
N_{Z/X}^*=\bigoplus_{\phi\in \hH} N_{Z,H,\phi}^*.
\end{equation}
Let $m_{\phi}$ be the rank of the subbundle $N_{Z,H,\phi}^*$, and let $\{x_{k_j}, 1\leq j\leq m_{\phi}\}$ be the Chern roots of $N_{Z,H,\phi}^*$. Then we have
\begin{equation}\label{eq: ch lambda_-1}
\ch_H\lambda_{-1}(N_{Z/X}^*) := \prod_{\phi\in \hH}\prod_{j=1}^{m_\phi} (1 -\phi e^{x_{k_j}}).
\end{equation}
For notational convenience, we may write this as
\[
\prod_{j=1}^{m} (1 -\phi_j e^{x_j}) 
\]
where \(\{x_j, 1\leq j\leq m\}\) are the Chern roots of \(N_{Z/X}^*\) with rank \(m\), and \(\{\phi_j, 1\leq j\leq m\}\) is the set of (possibly repeated) eigen characters associated to \(N_{Z/X}^*\) 
It is known that $\ch_H \lambda_{-1}(N_{Z/X}^*)$ becomes a unit in $H^{\even}(Z,\ZZ)\otimes R(H)_{h}$ by \cite[Lemma~2.7]{AS68II}, where $h$ is a generator of $H$.
We show that there exists a "partial inverse" $\tau_{Z,H}$ of $\ch_H\lambda_{-1}(N_{Z/X}^*)$ in $H^{\even}(Z,\QQ)\otimes R(H)_{\QQ}$ in Lemma~\ref{lem:quasi-inverse-class}.

For a stratum $Z$ and a cyclic group $H\in \sH_Z$, define the subset $K_{Z,H}\subset H$ by:
\begin{equation}\label{eq: K}
K_{Z,H}=
\begin{cases}
\{h\in H\mid \Tr(h;\phi_j)=1 \text{ for some $1\leq j\leq m$}\} & \text{if $Z$ is  not a component of $X$}\\
\{\id\} & \text{if $Z$ is a component of $X$}
\end{cases}
\end{equation}
When $Z$ is not a component of $X$, the set $K_{Z,H}$ consists exactly of those $h\in H$ for which $Z$ is not a component of $X^h$.
Define
\begin{equation}\label{eq: theta_ZH}
  \theta_{Z, H} =  [1_H] - \sum_{B\subset K_{Z,H}}\frac{|B|}{|H|} \Ind_B^H \theta_B\in R(H)_\QQ  
\end{equation}
where the summation ranges over cyclic groups $B$ contained in $K_{Z,H}$. Similar to $\theta_H$ in Definition~\ref{defn: theta_H}, \(\theta_{Z,H}\) is a characteristic module, whose support is the subset \(H\setminus K_{Z,H}\). In fact, for any \(h\in H\setminus K_{Z,H}\) and \(B\subset K_{Z,H}\), we have
\[
\theta_B(shs^{-1})=\theta_B(h)=0,\quad\forall\, s\in G.
\]
For any \(h\in K_{Z,H}\) with \(\phi_j(h)=1\) for some \(j\), the cyclic subgroup \(\langle h\rangle\) is contained in \(K_{Z,H}\) since \(\phi_j(h^k)=1\) for all \(k\). Then \(\theta_B(h)=1\) if and only if \(\langle h\rangle=B\). Hence, we have
\begin{equation}
\label{eqs:theta_ZH}
\Tr(h; \theta_{Z,H})=
\begin{cases}
    1 & \text{if $h\in H\setminus K_{Z,H}$;}\\
    0 & \text{if $h\in K_{Z,H}$.}
\end{cases}    
\end{equation}

\begin{rmk}
\begin{enumerate}
    \item For any $\phi\in R(H)$ with $\Supp(\phi)\subset \Supp(\theta_{Z,H})$ , one has $\phi\cdot \theta_{Z,H}=\phi$.
In particular, $\theta_{Z,H}^2=\theta_{Z,H}$, meaning $\theta_{Z,H}$ is idempotent. Similarly,  if $H$ is nontrivial, then $\theta_H\cdot\theta_{Z,H}=\theta_H$.
\item $\theta_{Z,H}$ is contained in the augmentation ideal $I_{H,\QQ}$.
\end{enumerate}
\end{rmk}

\begin{lem}
\label{lem:quasi-inverse-class}
Let $Z$ be a stratum of $(X,G)$ and $H\in \sH_Z$, with a generator $h_0$. Then there exists a unique element \(\tau_{Z,H}\in H^{\even}(Z,\QQ)\otimes R(H)\) that maps to $\theta_{Z,H}\ch_H\lambda_{-1}(N_{Z/X}^*)^{-1}$ under the localization map $H^{\even}(Z,\QQ)\otimes R(H)\rightarrow H^{\even}(Z,\QQ)\otimes R(H)_{h_0}$, and that $\Tr(h; \tau_{Z,H})=0$ for any $h\in K_{Z,H}$.
\end{lem}
\begin{proof}
Using \eqref{eq: ch lambda_-1}, we may write 
\begin{equation}\label{eq: tau_ZH}
\theta_{Z,H}\ch_H\lambda_{-1}(N_{Z/X}^*)^{-1} = \theta_{Z,H}\prod_{j=1}^{m} (1 -\phi_j e^{x_j}) = \prod_{j=1}^m\frac{\theta_{Z,H}}{1-\phi_j} \left(1-\frac{\theta_{Z,H}\phi_j}{1-\phi_j}\sum_{k\geq 1}\frac{x_j^k}{k!}\right)^{-1}    
\end{equation}
Expanding $\sU(s,t):=\left(1-s\sum_{k\geq1}\frac{t^k}{k!}\right)^{-1}$ as a formal series in $t$, we obtain 
\[
\sU(s,t)=\sum_{k\geq 0}A_k(s) t^k\in \QQ[s][[t]]
\]
where $A_k(s)\in\QQ[s]$ is a polynomial with rational coefficients for each $k$. Then
 \[
 \left(1-\frac{\theta_{Z,H}\phi_j}{1-\phi_j}\sum_{k\geq 1}\frac{x_j^k}{k!}\right)^{-1}=\sum_{k\geq 0}A_k\left(\frac{\theta_{Z,H}\phi_j}{1-\phi_j}\right) x_j^k
 \]
 and by \eqref{eq: tau_ZH}, we have
 \[
 \theta_{Z,H}\left(\ch_H\lambda_{-1}(N_{Z/X}^*)\right)^{-1}=\prod_{j=1}^m\frac{\theta_{Z,H}}{1-\phi_j} \sum_{k\geq 0}A_k\left(\frac{\theta_{Z,H}\phi_j}{1-\phi_j}\right) x_j^k
 \]
 By Lemma~\ref{lem: image to localization}, the term $\frac{\theta_{Z,H}}{1-\phi_j}$ are the image of some $\psi_j\in \frac{1}{|\phi_j|}R(H)$ satisfying $\psi_j(h)=0$ for any $h\in K_{Z,H}$. Define 
 \[
 \tau_{Z,H} = \prod_{j=1}^m\psi_j \sum_{k\geq 0}A_k\left(\psi_j\phi_j\right) x_j^k\in H^{\even}(Z,\QQ)\otimes R(H).
 \]
Then $\tau_{Z,H}$ maps to  $\theta_{Z,H}\ch_H\lambda_{-1}(N_{Z/X}^*)^{-1}$ by the localization map, and $\Tr(h; \tau_{Z,H}) =0$ for any $h\in K_{Z,H}$. For $h\in H\setminus K_{Z,H}$, we have 
\[
\Tr(h; \tau_{Z,H}) = \Tr(h; \ch_H\lambda_{-1}(N_{Z/X}^*)^{-1})\in H^{\even}(Z,\QQ).
\]
Therefore, $\Tr(h; \tau_{Z,H})$ is uniquely determined for any $h\in H$ and it follows that $\tau_{Z,H}$ is uniquely determined as an element of $H^{\even}(Z,\QQ)\otimes R(H)$.
\end{proof}
\begin{rmk}
As the proof of Lemma~\ref{lem:quasi-inverse-class} shows, the element $\left(\ch_H\lambda_{-1}(N_{Z/X}^*)\right)^{-1}$ in $ H^{\even}(Z,\QQ)\otimes R(H)_g$ does not have a preimage in $H^{\even}(Z,\QQ)\otimes R(H)_\QQ$ in general.
\end{rmk}

\begin{defn}
We define the \emph{ramification Todd class of $Z\subset X$ with respect to $H$} as
   \[
\td_H(Z) =  \td(Z)\cdot \tau_{Z,H}\in H^{\even}(Z,\QQ)\otimes R(H)
\]
\end{defn}
\begin{lem}\label{lem: ch_H td_H}
 Let $Z\in \sZ$ be a stratum.
\begin{enumerate}
    \item     If $\{\id\}\in \sH_Z$, then $Z$ is a component of $X$ and $\td_{\{\id\}}(Z)=0$.
    \item   For $H\subset H'$ in $\sH_Z$, we have 
    \[
   \Res_H^{H'} \ch_{H'}(\sE|_Z) = \ch_{H}(\sE|_Z),\quad \Res_H^{H'}\td_{H'} (Z)= \td_H(Z).
   \]
\end{enumerate}
\end{lem}
\begin{proof}
(i) Since $\{\id\}\in \sH_Z$, $Z$ is a component of $X^{\id}=X$. By definition, $\theta_{Z,\{\id\}} =0$, and hence $\td_{\{\id\}}(Z)=0$.

(ii) The first equality was proved in Lemma~\ref{lem: compatibility}. 
For the second equality, we first show that 
\(\Res^{H'}_H \theta_{Z,H'}=\theta_{Z, H}\).
As both are characteristic modules, it suffices to note that they have the same support, which is clear by definition \eqref{eq: K}:
\[
H\setminus K_{Z,H} = H\setminus K_{Z,H'}.
\]

Now we prove that \(\Res^{H'}_H \tau_{Z,H'}=\tau_{Z, H}\). The equality \(\Res_H^{H'}\td_{H'} (Z)= \td_H(Z)\) then follows directly. The element \(\tau_{Z,H}\), supported on \(H\setminus K_{Z,H}\), is uniquely determined by the equation 
\[
\tau_{Z,H}\cdot \ch_H\lambda_{-1}(N_{Z/X}^*)=\theta_{Z,H}.
\]
Through the above discussion and Lemma~\ref{lem: compatibility} we have
\begin{align*}
\theta_{Z,H}=&\Res^{H'}_{H}\theta_{Z,H'}\\
=&\Res^{H'}_{H}(\tau_{Z,H'})\cdot \Res^{H'}_{H}(\ch_{H'}\lambda_{-1}(N_{Z/X}^*))\\
=&\Res^{H'}_{H}(\tau_{Z,H'})\cdot \ch_H\lambda_{-1}(\Res^{H'}_{H}N_{Z/X}^*)).
\end{align*}
Since \(\Res^{H'}_{H}(\tau_{Z,H'})\) is also supported on \(H\setminus K_{Z,H}\), we conclude that \(\Res^{H'}_{H}(\tau_{Z,H'})=\tau_{Z, H}\).
\end{proof}

The following definition of ramification module generalizes \cite[Definition~3.2]{LL25}, originally conceived for smooth projective curves.
\begin{defn}\label{defn: Gamma_Z}
For each stratum $Z\in \sZ$, we define the \emph{ramification module} of the $G$-manifold $X$ at $Z$ as
\begin{equation}\label{eq: Gamma_Z}
\Gamma(\sE)_Z:=\sum_{H\in \sH_Z} \frac{|H|}{|G|}\Ind_H^G \left(\theta_H 
\int_Z\ch_H(\sE|_Z)\td_H(Z)\right)
\end{equation}
\end{defn}
\begin{lem}\label{lem: Gamma_Z 0}
Let $Z\in \sZ$ be a stratum. Then the following holds for $\Gamma(X)_Z$:
\begin{enumerate}
    \item $\Gamma(\sE)_Z$ is an element of the augmentation ideal $I_{G,\QQ}$. 
    \item If $G_Z=\{\id\}$, then $Z$ is a component of $X$ and $\Gamma(X)_Z=0$.
\end{enumerate}
\end{lem}
\begin{proof}
    (i) By Lemma~\ref{eqs:theta_ZH}, $\Tr(1;\theta_{Z,H})=0$. It follows that the traces of $g$ on $\tau_{Z,H}$, $\td_H(Z)$, and $\Gamma(\sE)_Z$ are all zero. 

(ii) Note that $\sH_Z$ is always nonempty and its elements are subgroups of $G_Z$. If $G_Z=\{\id\}$, then $\sH_Z=\{G_Z\}$. It follows that $Z$ is a component of $X=X^\id$, and $\td_{G_Z}(Z)=0$ by Lemma~\ref{lem: ch_H td_H}. Therefore, $\Gamma(X)_Z=0$.
\end{proof}
\begin{thm}\label{thm: CW general}
Under Notation~\ref{nota: CW}, the following equality holds in $R(G)_\QQ$:
\begin{equation}
\label{eq: CW general}
\chi_G(X, \sE) = \frac{1}{|G|}\chi(X, \sE)[\CC[G]] + \sum_{Z\in \sZ} \Gamma(\sE)_Z    
\end{equation}
\end{thm}
\begin{proof}
By Proposition~\ref{prop: recover chi}, we have
\begin{equation}
\label{eq: thm lhs}
\chi_G(X, \sE) = \sum_{H\subset G \text{ cyclic}} \frac{|H|}{|G|}\Ind_H^G(\theta_H\cdot\Res_H^G\chi_G(X, \sE)) 
\end{equation}
By the definition of \(\theta_H\), the class function \(\theta_H\cdot\Res_H^G\chi_G(X, \sE)\)
is determined by the values on the generators $h$ of the group \(H\). By Theorem~\ref{thm: AS} and the definitions of $\ch_H$ and $\td_H$, if $H$ is nontrivial, then
\begin{equation}
    \begin{split}
\Tr(h;\theta_H\cdot\Res_H^G\chi_G(X, \sE))& =\Tr\left(h;\theta_{H}\sum_{Z\in \mathcal{Z}_H}\int_Z\frac{\ch_H(\sE|_Z)\cdot \td(Z)}{\ch_H\lambda_{-1}(N^*_{Z/X})}\right)\\
&=\Tr\left(h;\theta_{H}\theta_{Z,H}\sum_{Z\in \mathcal{Z}_H}\int_Z\frac{\ch_H(\sE|_Z)\cdot \td(Z)}{\ch_H\lambda_{-1}(N^*_{Z/X})}\right)\\
&=\Tr\left(h;\sum_{Z\in \mathcal{Z}_H}\int_Z\theta_H\ch_H(\sE|_Z)\tau_{Z,H}\td(Z)\right) \\
&= \Tr\left(h;\sum_{Z\in \mathcal{Z}_H}\int_Z\theta_H\ch_H(\sE|_Z)\td_H(Z)\right).
    \end{split}
\end{equation}
Hence, \(\theta_H\Res_H^G\chi_G(X, \sE)\) and $\int_Z\theta_H\ch_H(\sE|_Z)\td_H(Z)$
represent the same class in \(R(H)_\QQ\). Thus, the right hand side of \eqref{eq: thm lhs} can be written as
\begin{equation}\label{eq: thm lhs2}
\frac{1}{|G|}\Ind_{\{\id\}}^G(\theta_{\{\id\}}\cdot\Res_{\{\id\}}^G\chi_G(X, \sE))+ \sum_{\{\id\}\subsetneq H\subset G \text{ cyclic}} \frac{|H|}{|G|}\Ind_H^G(\theta_H\cdot\Res_H^G\chi_G(X, \sE)).
\end{equation}
The first terms simplifies to $\frac{1}{|G|}\chi(X, \sE)[\CC[G]]$. The second term becomes
\[
\sum_{\{\id\}\subsetneq H\subset G \text{ cyclic}} \frac{|H|}{|G|}\Ind_H^G\sum_{Z\in \sZ_H}\int_Z\theta_H\ch_H(\sE|_Z)\td_H(Z).
\]
Reordering the summation gives
\begin{equation}\label{eq: thm lhs3}
\sum_{Z\in \sZ}\sum_{\{\id\}\neq H\in \sH_Z} \frac{|H|}{|G|}\Ind_H^G\int_Z\theta_H\ch_H(\sE|_Z)\td_H(Z) = \sum_{Z\in \sZ} \Gamma(\sE)_Z,
\end{equation}
where we have used Definition~\ref{defn: Gamma_Z} and Lemma~\ref{lem: Gamma_Z 0}.  Substituting \eqref{eq: thm lhs3} into \eqref{eq: thm lhs2} yields the required equality \eqref{eq: CW general}.
\end{proof}

\begin{rmk}
Theorem~\ref{thm: CW general} recovers \cite[Theorem~3.6]{LL25} in the case of (possibly disconnected) complex smooth projective curves.
\end{rmk}
 
Using \cite[5.5]{Don69} instead of Theorem~\ref{thm: AS}, one obtains an algebraic version of Theorem~\ref{thm: CW general}.
\begin{thm}
Let $k$ be an algebraically closed field of characteristic $p\geq 0$ and $X$ a smooth proper variety over $k$. Let $G$ be a finite group such that $p\nmid |G|$, acting on $X$. Then we may define the set $\sZ$ of strata as in Definition~\ref{defn: strata}, and for each stratum $Z\in \sZ$ and a locally free $G$-sheaf $\sE$ on $X$, a ramification module $\Gamma(\sE)_Z$, 
depending only on the restriction $\sE|_Z$ as a $G_Z$-sheaf,
such that the following holds in $R_k(G)_\QQ$:
\[
\chi_G(X, \sE) = \frac{1}{|G|}\chi(X, \sE)[\CC[G]] + \sum_{Z\in \sZ} \Gamma(\sE)_Z.
\]
\end{thm}

Using the formula~\eqref{eq: CW general}, we can generalize the classical Chevalley--Weil formula on curves, which describes the multiplicity of any irreducible representation in the virtual \(G\)-module \(\chi_G(X, \sE)\) in terms of the ramification data.
\begin{cor}\label{cor: multiplicity}
Under Notation~\ref{nota: CW}, for each irreducible $M\in \hG$, its multiplicity in $\chi_G(X, \sE)$ is given by
\[
\mu_M\chi_G(X, \sE) = \frac{\dim M}{|G|}\chi(X,\sE)+\sum_{Z\in \sZ} \sum_{H\in \sH_Z}\frac{|H|}{|G|}\left\langle \Res^G_HM, \int_Z\theta_H \ch_H(\sE|_Z)\td_H(Z)\right\rangle_H.
\]
\end{cor}
\begin{proof}
By Theorem~\ref{thm: CW general}, we have
\[
\mu_M\chi_G(X, \sE)= \langle  [M], \chi_G(X, \sE) \rangle_G  = \langle [M], \frac{1}{|G|}\chi(X, \sE)[\CC[G]]\rangle_G + \sum_{Z\in \sZ} \langle [M],\Gamma(\sE)_Z  \rangle_G.
\]
The first term equals $\frac{\dim M}{|G|}\chi(X,\sE)$ by a standard property of the regular representation $\CC[G]$. For the second term, using the Frobenius reciprocity and the definition of $\Gamma(\sE)_Z$ yields
\[
\langle [M],\Gamma(\sE)_Z  \rangle_G =\sum_{H\in\sH_Z} \frac{|H|}{|G|}\left\langle \Res^G_H[M], \int_Z\theta_H \ch_H(\sE|_Z)\td_H(Z)\right\rangle_H
\]
for any stratum $Z\in \sZ$.
\end{proof}

\section{Computation of the ramification module in special cases}\label{sec: ramification module}
The practical utility of the Chevalley--Weil formula \eqref{eq: CW general} often depends on the computability of the ramification modules $\Gamma(\sE)_Z$. In the following, we illustrate how $\Gamma(\sE)_Z$ can be simplified under restrictions on the stabilizer $G_Z$, or on the dimension and codimension of $Z$.

If the stabilizer $G_Z$ is cyclic, then $\Gamma(\sE)_Z$ simplifies, as one does not need to sum over all of $\sH_Z$.
\begin{lem}\label{lem: cyclic intertia}
Suppose that $Z\in \sZ$ is a stratum with cyclic stabilizer $G_Z$. 
Then
\begin{equation}
\label{eq: Gamma_Z G_Z cyclic}
\Gamma(\sE)_Z =\frac{|G_Z|}{|G|}\Ind_{G_Z}^G \int_Z\ch_{G_Z}(\sE|_Z)\td_{G_Z}(Z).
\end{equation}
\end{lem}
\begin{proof}
    We first claim that \(G_Z\in \mathcal{H}_Z\). 
    This is clear if $Z$ is a component of $X$. 
    If \(Z\) is not a component in the fixed locus \(X^{G_Z}\), \(G_Z\) acts identically along some normal direction \(v\in N_{Z/X}\). Consequently, any cyclic subgroup of \(G_Z\) acts identically along \(v\), which contradicts that \(Z\), as a stratum, is a component in \(X^H\) for some subgroup \(H\subset G_Z\).     
    
    Using Lemma~\ref{lem: ch_H td_H}, we have
\begin{equation}\label{eq: Gamma_Z G_Z cyclic 2}
\begin{split}
\Gamma(\sE)_Z & = \sum_{H\in \sH_Z}\frac{|H|}{|G|}\Ind_H^G \int_Z\theta_H\ch_H(\sE|_Z)\td_H(Z) \\
  & = \frac{|G_Z|}{|G|}\Ind_{G_Z}^G   \sum_{H\in \sH_Z}\frac{|H|}{|G_Z|}\Ind_H^{G_Z} \theta_H \Res_H^{G_Z}\int_Z\ch_{G_Z}(\sE|_Z)\td_{G_Z}(Z).
\end{split}
\end{equation}

Let $K:=\{g\in G_Z\mid \text{$Z$ is not a component of $X^g$}\}$. Then
$\sH_Z$ consists of subgroups of $G_Z$ that are not contained in $K$. For a cyclic subgroup $H\subset G_Z$ contained in $K$, we have $\Res^{G_Z}_H\td_{G_Z}(Z)=0$ and hence $\Res_H^{G_Z}\int_Z\ch_{G_Z}(\sE|_Z)\td_{G_Z}(Z)=0$.

Thus, the last summation of \eqref{eq: Gamma_Z G_Z cyclic 2} is the same as summing over all the subgroups $H\subset G_Z$:
\begin{align*}
    \sum_{H\subset G_Z}\frac{|H|}{|G_Z|}\Ind_H^{G_Z} \theta_H \Res_H^{G_Z}\int_Z\ch_{G_Z}(\sE|_Z)\td_{G_Z}(Z)
= \int_Z\ch_{G_Z}(\sE|_Z)\td_{G_Z}(Z) 
\end{align*}
where the equality follows from Proposition~\ref{prop: recover chi}. Substituting this into \eqref{eq: Gamma_Z G_Z cyclic 2} yields the required formula \eqref{eq: Gamma_Z G_Z cyclic}.
\end{proof}

\begin{ex}\label{ex: codim 0}
Suppose that $Z\in \sZ$ is a connected component of $X$. Then 
\[
\sH_Z=\{H\mid H\subset G_Z \text{ cyclic}\}.
\]
The ramification module $\Gamma(\sE)_Z$ can be computed as follows:
\begin{align*}
\Gamma(\sE)_Z & =\sum_{H\in \sH_Z} \frac{|H|}{|G|}\Ind_H^G \int_Z\theta_H\ch_H(\sE|_Z)\td_H(Z) \\
& =\frac{|G_Z|}{|G|}\Ind^{G}_{G_Z}\sum_{H\in \sH_Z} \frac{|H|}{|G_Z|}\Ind_H^{G_Z} \int_Z\theta_H \td_H(Z)\cdot \Res_H^{G_Z} \ch_{G_Z}(\sE|_Z)) \\
& = \frac{|G_Z|}{|G|}\Ind^{G}_{G_Z}\sum_{H\in \sH_Z} \frac{|H|}{|G_Z|}\int_Z\ch_{G_Z}(\sE|_Z)\cdot\Ind_H^{G_Z} (\theta_H \td_H(Z))\\
\end{align*}
Recall that \(\td_H(Z)=\td(Z)\) if \(H\neq \{\id\}\), and \(\td_{\{\id\}}(Z)=0\). 
Then the last term equals
\[
\frac{|G_Z|}{|G|}\Ind^{G}_{G_Z}\int_Z\ch_{G_Z}(\sE|_Z)\td(Z)\left(\sum_{H\subset G_Z ~\text{cyclic}}\frac{|H|}{|G_Z|}\Ind_H^{G_Z} \theta_H-\frac{1}{|G_Z|}\Ind^{G_Z}_{\{\id\}}\theta_{\{\id\}}\right)
\]
By the Hirzebruch--Riemann--Roch theorem and Formula \eqref{eq: recover chi}, this becomes
\[
\frac{|G_Z|}{|G|}\Ind^{G}_{G_Z}\chi_{G_Z}(Z,\sE|_Z)\left([1_{G_Z}]-\frac{1}{|G_Z|}\Ind^{G_Z}_{\{\id\}}\theta_{\{\id\}}\right)
\]
Simplifying yields
\begin{align*}
&\frac{|G_Z|}{|G|}\Ind^{G}_{G_Z}\chi_{G_Z}(Z,\sE|_Z)-\frac{1}{|G|}\Ind^{G}_{G_Z}\Ind^{G_Z}_{\{\id\}}(\theta_{\{\id\}}\cdot \Res^{G_Z}_{\{\id\}}\chi_{G_Z}(Z,\sE|_Z))\\
=&\frac{|G_Z|}{|G|}\Ind^{G}_{G_Z}\chi_{G_Z}(Z,\sE|_Z)-\frac{\chi(Z,\sE|_Z)}{|G|}[\mathbb{C}[G]]
\end{align*}
where \(\chi_{G_Z}(Z,\sE|_Z)\) represents the \(G_Z\)-Euler characteristic of \(\sE|_Z\),
and \(\chi(Z, \sE|_Z)\) is the usual Euler characteristic.
\if false Now note that, by \cite[Proposition~27]{Ser77},
 \[
 \sum_{H\subset K \text{ cyclic}} \Ind_H^K |H|\cdot\theta_H = |K|\cdot [1_K],
 \]
so we deduce that
\begin{equation}\label{eq: Gamma(E)_Z}
\Gamma(\sE)_Z = \frac{|K|}{|G|}\chi(Z, \sE) \Ind_K^G [1_K]
\end{equation}
\fi

If $G_Z=\{\id\}$ for each component $Z$ of $X$, then $\sH_Z$ consists of only the trivial group, and we have $\Gamma(\sE)_{Z}=0$.
In this case, \eqref{eq: CW general} becomes
\begin{equation}\label{eq: CW faithful 2}
\chi_G(X, \sE) =\frac{1}{|G|}\chi(X,\sE)[\CC[G]]+ \sum_{Z'\in \sZ'} \Gamma(\sE)_{Z'}.  
\end{equation} 
where $\sZ'\subset \sZ$ consists of the strata with positive codimension in $X$.
\end{ex}

\begin{ex}
\label{ex: dim 0}
Suppose that $P$ is an isolated point in $X^g$ for some $g\in G$ and $\codim_{X}\{P\}>0$. Then $\{P\}\in \sZ$ and $H:=\langle g\rangle\in \sH_{P}$. 
We have $\td(\{P\})=1$. Decompose the conormal space $N_{Z/X}^*=T_P^*X$ into eigenspaces 
$T_P^*X=\bigoplus_{\phi\in \hH}V_{\phi}$, and let \(m_{\phi}\) denote the dimension of \(V_{\phi}\). Then
\[
\td_{H}(Z) = \prod_{\phi\in \hH} \left(-\frac{1}{|\phi|}\sum_{d=0}^{|\phi|-1} d\phi^d\right)^{m_{\phi}}
\]
where \(|\phi|\) is the order of \(\phi\). Also,
\[
\ch_{H}(\sE|_P) = [\sE|_P]_H\in R(H).
\]
Therefore,
\[
\Gamma_G(\sE)_P = \sum_{H\in \sH_P} \frac{|H|}{|G|}\Ind_H^G\theta_H [\sE|_P]_H  
\prod_{\phi\in \hH} \left(-\frac{1}{|\phi|}\sum_{d=0}^{|\phi|-1} d\phi^d\right)^{m_{\phi}}
\]
\end{ex}

\begin{ex}
\label{ex: codim 1}
Suppose that $X$ is connected and $G$-action on $X$ is faithful. 
Let $Z\in \sZ$ be a stratum with $\codim_X Z=1$. Then $G_Z$ is cyclic. 
The conormal bundle $N^*_{Z/X}$ has rank $1$, 
and $G_Z$ acts on it by a character $\phi_Z\colon G_Z\rightarrow \CC^*$. 
Since the codimension of \(Z\) is the smallest among all strata except \(X\), 
$\sH_Z$ consists of all the nontrivial subgroup of $G_Z$ and the subset $K_{Z,G_Z}$ in \eqref{eq: K} is the identity. Therefore, $\theta_{Z,G_Z} = [1_{G_Z}]-\frac{1}{|G_Z|}[\CC[G_Z]]$ (see \eqref{eq: theta_ZH} for the definition of $\theta_{Z,H}$ with $H\in \sH_Z$).
By Lemma~\ref{lem: cyclic intertia}, we have
\begin{equation}\label{eq: codim 1}
\Gamma(\sE)_Z  =\frac{|G_Z|}{|G|}\Ind_{G_Z}^G\sum_{\psi\in\hG_Z}\psi \int_Z\ch(\sE_{Z,\psi})\td_{G_Z}(Z)
\end{equation}
where $[\sE|_Z]_{G_Z}=\sum_{\psi\in\hG_Z} [\sE_{Z,\psi}]\otimes [\psi] \in K(Z)\otimes R(G_Z)$.
\end{ex}

\section{The action of \texorpdfstring{$(\ZZ/2\ZZ)^n$}{} on a compact complex surface}\label{sec: surface}
We illustrate the general Chevalley--Weil formula of Theorem~\ref{thm: CW general} by working out the case where $G\cong (\ZZ/2\ZZ)^n$ acts on a connected compact complex surface.

First, we refine the expression of the ramification module from Example~\ref{ex: codim 1} at a fixed curve on a surface.
\begin{lem}\label{lem: Gamma_Z curve}
Let $X$ be a connected compact complex surface, $G\subset \Aut(X)$ a finite subgroup, and $\sE$ a $G$-equivariant locally free sheaf on $X$. Suppose that $C$ is a one-dimensional stratum of $(X, G)$. Denote by \(G_C\) the stabilizer group of \(C\), and by \(\phi_C\colon G_C\to \CC^*\) the character of the rank \(1\) conormal bundle \(N^*:=N^*_{C/X}\) under the action of \(G_C\). Then
    \[
    \Gamma(\sE)_C =- \frac{|G_C|}{|G|}\Ind_{G_C}^G \theta_{C,G_C}\sum_{\psi\in \hG_C}\psi\left(\frac{\chi(\sE_{C,\psi})}{|G_C|}\sum_{d=0}^{|G_C|-1}d\phi_C^d + \frac{(\rk\, \sE_{C,\psi})\cdot C^2}{|G_C|^2}\left(\sum_{d=0}^{|G_C|-1}d\phi_C^d \right)^2\phi_C \right)
    \]
    where $\theta_{C,G_C} = [1_{G_C}]-\frac{1}{|G_C|}[\CC[G_C]]$.
\end{lem}
\begin{proof}
Since $\dim X=2$ and $C$ is a curve, Example~\ref{ex: codim 1} gives $\theta_{C, G_C} =[1_{G_C}]-\frac{1}{|G_C|}[\CC[G_C]]$, and 
\begin{equation}
\begin{split}
\tau_{C, G_C} &= \frac{\theta_{C,G_C}}{1-\phi_C\cdot e^{c_1(N^*)}} = \frac{\theta_{C, G_C}}{1-\phi_C} + \frac{\theta_{C,G_C}\phi_C}{(1-\phi_C)^2} c_1(N^*)\\
&=-\frac{\theta_{C,G_C}}{|G_C|}\sum_{d=0}^{|G_C|-1}d\phi_C^d + \frac{\theta_{C,G_C}\phi_C}{|G_C|^2}\left(\sum_{d=0}^{|G_C|-1}d\phi_C^d \right)^2c_1(N^*).
\end{split}
\end{equation}
We may write
\[
\ch_{G_C}(\sE_{C, \psi}) = \left(\rk\, \sE_{C,\psi} + c_1(\sE_{C,\psi})\right)\psi.
\]
It follows that
\[
\begin{split}
 &\int_C\ch(\sE_{C,\psi})\td_{G_C}(C)\\
 =&\int_C\left(\rk\, \sE_{C,\psi} + c_1(\sE_{C,\psi}) \right)\left(\frac{\theta_{C,G_C}}{1-\phi_C} + \frac{\theta_{C,G_C}\phi_C}{(1-\phi_C)^2} c_1(N^*)\right)\left(1+\frac{1}{2}c_1(T_C)\right) \\
=&\int_C\theta_{C,G_C}\left(\frac{\ch(\sE_{C,\psi})\cdot \td(C)}{1-\phi_C}+ \frac{(\rk\, \sE_{C,\psi})\cdot \phi_C}{(1-\phi_C)^2}c_1(N^*)\right).
\end{split}
\]
Note that \(\int_{C}\ch(\sE_{C,\psi})\cdot \td(C)=\chi(\sE_{C,\psi})\) and \(c_1(N^*)=-C^2\). Hence, the above expression equals 
\begin{equation}
\label{eq: dim Z=1 integrand}
\begin{split}
&\theta_{C,G_C}\left(\frac{\chi(\sE_{C,\psi})}{1-\phi_C} - \frac{\rk\, \sE_{C,\psi}\cdot\phi_C}{(1-\phi_C)^2}C^2\right) \\
=&-\theta_{C,G_C}\left(\frac{\chi(\sE_{C,\psi})}{|G_C|}\sum_{d=0}^{|G_C|-1}d\phi_C^d +\frac{\rk\, \sE_{C,\psi}}
{|G_C|^2}\phi_C\left(\sum_{d=0}^{|G_C|-1}d\phi_C^d \right)^2 C^2\right).
\end{split}
\end{equation}
Substituting \eqref{eq: dim Z=1 integrand} into \eqref{eq: codim 1} yields the desired equality.
\end{proof}

\begin{thm}\label{thm: Z2n}
    Let $X$ be a connected compact complex surface and $G\subset\Aut(X)$ an automorphism subgroup isomorphic to $(\ZZ/2\ZZ)^n$. Let $\sE$ be a $G$-equivariant locally free sheaf of rank $r$ on $X$. Suppose that there are $N$ zero-dimensional strata $P_1,\dots, P_N$ and $M$ one-dimensional strata $C_1,\dots, C_M$ for the $G$-action on $X$. Let $\sE|_{C_k} =\sE_{C_k}^+ \oplus \sE_{C_k}^-$ be the decomposition of $\sE|_{C_k}$ into eigensheaves with eigenvalues $1$ and $-1$ under the action of $G_{C_k}$, and let $r_k^+$ and $r_k^-$ be the ranks of $\sE_{C_k}^+$ and $\sE_{C_k}^-$ respectively.  Then the following holds.
    \begin{enumerate}[leftmargin=*]
        \item For each strata $Z\neq X$, $\sH_Z$ has exactly one element, denoted by $H_Z$.
        \item We have
        \begin{multline*}
        \chi_G(X, \sE) = \frac{1}{2^n}\chi(X, \sE)[\CC[G]] + \frac{1}{2^{n+1}} \sum_{i=1}^N \Ind_{H_{P_i}}^G \left([\sE|_{P_i}] - \frac{r}{2}[\CC[H_{P_i}]]\right)\\
        + \frac{1}{2^{n+1}}\sum_{k=1}^M\Ind_{H_{C_k}}^G\left(-(K_X\cdot C_k)(r_k^+-r_k^-)+2(\deg\sE_{C_k}^+-\deg\sE_{C_k}^-)\right)\left([1_{H_{C_k}}]-\frac{1}{2}[\CC[H_{C_k}]]\right)
        \end{multline*}
    \end{enumerate}
    \end{thm}
\begin{proof}
(i) If $\dim Z=1$, then $G_Z$ is cyclic and hence has order 2. Clearly, $\sH_Z=\{G_Z\}$. If $\dim Z=0$, then $|G_Z|\leq 2^2$ by \cite[Lemma~2.1]{CL24}. By \cite[Lemma~2.7]{CL24}, there is exactly one involution $\sigma\in G_Z$ having $Z$ as an isolated fixed point, so $\sH_Z=\{\langle\sigma\rangle\}$. 

(ii) Based on (i), we have $\theta_{H_Z} =\theta_{Z,H_Z} =[1_{H_Z}]-\frac{1}{2}[\CC[H_Z]]$. We now compute $\Gamma(\sE)_Z$ for each stratum of $(X, G)$. Denote by $\phi_Z$ the class of the nontrivial simple module of $H_Z$.

Since $G$-action on $X$ is faithful, we have $\Gamma(\sE)_Z=0$ if $Z=X$; see Example~\ref{ex: codim 0}.

If $Z= C_k$ for some $1\leq k\leq M$, then by Lemma~\ref{lem: Gamma_Z curve}, we have
\begin{equation}
    \begin{split}
        \Gamma(\sE)_{Z} &        
= \frac{1}{2^{n-1}}\Ind_{H_Z}^G
            ([1_{H_Z}]-\frac{1}{2}[\CC[H_Z]]) 
            \left(\left(-\frac{1}{2}\chi(\sE_Z^+)-\frac{1}{4}r^+Z^2\right)\phi_Z +\left(-\frac{1}{2}\chi(\sE_Z^-)-\frac{1}{4}r^-Z^2\right)\right)\\
        &= \frac{1}{2^{n+1}}\Ind_{H_{Z}}^G\left((r_k^+-r_k^-)\cdot Z^2+2(\chi(\sE_{Z}^+)-\chi(\sE_{Z}^-))\right)\left([1_{H_{Z}}]-\frac{1}{2}[\CC[H_{Z}]]\right) \\
        & = \frac{1}{2^{n+1}}\Ind_{H_{Z}}^G\left(-(K_X\cdot Z)(r_k^+-r_k^-)+2(\deg\sE_{Z}^+-\deg \sE_{Z}^-)\right)\left([1_{H_{Z}}]-\frac{1}{2}[\CC[H_{Z}]]\right) 
    \end{split}
\end{equation}
where we used the Riemann--Roch theorem $\chi(Z, \sF)=(\rk \sF)\chi(\sO_Z)+\deg\sF$ for a locally free sheaf $\sF$ on $Z$, the adjuction formua $(K_X+Z)\cdot Z+2\chi(Z,\sO_Z)=0$, and the identity
\[
\left([1_{H_Z}]-\frac{1}{2}[\CC[H_Z]]\right) \phi_Z = -  \left([1_{H_Z}]-\frac{1}{2}[\CC[H_Z]]\right).
\]

If $Z = P_i$ for some $1\leq i\leq N$, then by Example~\ref{ex: dim 0}
\begin{equation}
\begin{split}
 \Gamma(\sE)_{Z} & = \frac{|H_{Z}|}{|G|}\Ind_{H_{Z}}^G \theta_{H_{Z}}[\sE|_{Z}]\left(\frac{1}{2}\phi_{Z}\right)^2   \\ &= \frac{1}{2^{n-1}}\Ind_{H_{Z}}^G\frac{1}{4}\left([1_{H_{Z}}]-\frac{1}{2}[\CC[H_{Z}]]\right)[\sE|_{Z}] \\
 &=\frac{1}{2^{n+1}}\left([\sE|_{Z}] - \frac{r}{2}[\CC[H_{Z}]]\right).
\end{split}
\end{equation}

Summing $\Gamma(\sE)_Z$ over all the strata $Z$ and applying Theorem~\ref{thm: CW general} yields the desired equality in (ii).
\end{proof}

Applying Theorem~\ref{thm: Z2n} to the sheaves $\sE=\sO_X(nK_X)$ for $n\in \ZZ$, and $\Omega_X^1$, we obtain more explicit results.

\begin{ex}
Take $\sE=\Omega_X^1$ in Theorem~\ref{thm: Z2n}. For a zero-dimensional stratum $Z= P_i$, we have 
\[
[\Omega_X^1|_{Z}] = [T_Z^*X] = 2[\phi_Z]
\]

For a one-dimensional stratrum $Z=C_k$,  the short exact sequence $0\rightarrow N_{Z/X}^*\rightarrow \Omega_X^1|_Z\rightarrow \Omega_Z^1\rightarrow 0$ is $H_Z$-equivariant. Here, $H_Z$ acts on $N_{Z/X}^*$ by the nontrivial character $\phi_Z$ and on $\Omega_Z^1$ by the trivial character $1_Z$. Thus the sequence gives an eigen-subsheaf decomposition: 
\[
\Omega_X^1|_Z = N_{Z/X}^*\oplus \Omega_Z^1
\]
In the notation of Theorem~\ref{thm: Z2n}, for each $1\leq k\leq M$, 
\[
r_k^+= \rk\, \Omega_Z^1=1,\quad r_k^-= \rk\, N_{Z/X}^*=1, \quad \deg N_{Z/X}^*=-Z^2,\quad \deg \Omega_Z^1 = 2g(Z)-2
\]

Using the exact sequence $0\rightarrow \CC\xrightarrow{\bar\partial} \sO_X\xrightarrow{\bar\partial}\Omega_X^1\xrightarrow{\bar\partial} \Omega_X^2\rightarrow 0$, we find
\[
\chi(X, \Omega_X^1)= -\chi(X, \CC) +\chi(X,\sO_X)+\chi(X, \Omega_X^2)= K_X^2-10\chi(X, \sO_X)
\]
where we use the Serre duality $\chi(X, \Omega_X^2) = \chi(X,\sO_X)$ and  the Noether formula $\chi(X, \CC)=12\chi(\sO_X)-K_X^2$ for the second equality. Alternatively, one may apply the Hirzebruch--Riemann--Roch theorem directly.

By Theorem~\ref{thm: Z2n} (ii), we obtain, 
\begin{multline*}
        \chi_G(X, \Omega_X^1) = \frac{1}{2^n}(K_X^2-10\chi(\sO_X))[\CC[G]] -\frac{1}{2^{n}} \sum_{i=1}^N \Ind_{H_{P_i}}^G \left([1_{H_{P_i}}] -\frac{1}{2}[\CC[H_{P_i}]]\right)\\
        + \frac{1}{2^{n}}\sum_{k=1}^M\Ind_{H_{C_k}}^G(2g(C_k)-2+
        C_k^2)\left([1_{H_{C_k}}]-\frac{1}{2}[\CC[H_{C_k}]]\right)
        \end{multline*}
\end{ex}

\begin{ex}
Now take $\sE=\sO_X(nK_X)$, $n\in \ZZ$. This sheaf is invertible.

If $Z=P_i$ is a point, then the action near $Z$ can be linearized analytically locally as $(x,y)\mapsto (-x,-y)$, which fixes the local basis $(dx\wedge dy)^{\otimes n}$ of $\sO_X(nK_X)$. Thus,
\[
[\sO_X(nK_X)|_Z] = [1_{H_Z}]
\]

If $Z= C_k$ is a curve, then there is analytically local coordinates $x,y$ so that $Z=(x=0)$ and the involution in $H_Z$ acts as $(x,y)\mapsto(-x,y)$. Therefore, $H_Z$ acts on $\sO_X(nK_X)$ by the character $\phi_Z^n$. Hence, we have 
\[
r_k^+ - r_k^- = (-1)^n, \quad \deg\sO_X(nK_X)_Z^+ - \deg\sO_X(nK_X)_Z^- = (-1)^nn(K_X\cdot Z)
\]

Also, for $Z=X$, the Riemann--Roch theorem gives
\[
\chi(X, \sO_X(nK_X)) = \chi(X, \sO_X) + \frac{1}{2}n(n-1)K_X^2
\]

Substituting these computations into Theorem~\ref{thm: Z2n} (ii) with $\sE=\sO_X(nK_X)$ yields:
\begin{multline*}
\chi_G(X, nK_X) = \frac{1}{2^n}\left(\chi(X, \sO_X) + \frac{1}{2}n(n-1)K_X^2\right)[\CC[G]] + \frac{1}{2^{n+1}}\sum_{1\leq i\leq N}\Ind_{H_{P_i}}^G ([1_{H_{P_i}}]-\frac{1}{2}[\CC[H_{P_i}]]) \\
+ (-1)^n\frac{2n-1}{2^{n+1}}\sum_{1\leq k\leq M}\Ind_{H_{C_k}}^G (K_X\cdot C_k)([1_{H_{C_k}}]-\frac{1}{2}[\CC[H_{C_k}]]) .  
\end{multline*}
\end{ex}

\end{document}